\documentclass{amsart}
\usepackage{amssymb,amsmath,ytableau,youngtab,young,tikz,pgf,physics}
\newtheorem{theorem}{Theorem}
\newtheorem{corollary}{Corollary}
\newtheorem{lemma}{Lemma}

\newtheorem{proposition}{Proposition}

\begin{document}

\title[Colored partitions and the hooklength formula]{Colored partitions and the hooklength formula: partition statistic identities}

\author[Anible \& Keith]{Emily E. Anible, William J. Keith}

\begin{abstract}

We give relations between the joint distributions of multiple hook lengths and of frequencies and part sizes in partitions, extending prior work in this area.  These results are discovered by investigating truncations of the Han/Nekrasov-Okounkov hooklength formula and of $(k,j)$-colored partitions, a unification of $k$-colored partitions and overpartitions.  We establish the observed relations at the constant and linear terms for all $n$, and for $j=2$ in their quadratic term, with the associated hook/frequency identities.  Further results of this type seem likely.
\end{abstract}

\maketitle

\section{Introduction}\label{sec:Introduction}

Identities relating the number of parts of a given size in a partition to other statistics, such as the number of repetitions of part sizes or the number of hooks of a given length, are of classical interest in the theory of partitions.  Two known identities are the following.

\begin{proposition} \cite{Fine} The total number of parts of size $i$ appearing in all partitions of $n$ equals the number of times a part is repeated at least $i$ times in all partitions of $n$.
\end{proposition}

\begin{proposition} \cite['98]{Bessen}, \cite['02]{BM} The number of hooks of length $i$ appearing in all partitions of $n$ is equal to $i$ times the total number of parts of size $i$ appearing in all partitions of $n$.
\end{proposition}

Bacher \& Manivel \cite{BM} give further relations between powers of parts and frequencies of appearance of part sizes. They define vectors $\lambda$, $\nu$, and $\gamma$, in which $\lambda_k$ is the $k$-th part of partition $\lambda$, $\nu_k$ is the multiplicity of part size $k$, and $\gamma_{\geq k}$ is the number of part sizes in $\lambda$ repeated at least $k$ times. Among other results they give generating functions for $\binom{\lambda_k}{d}$ and $\binom{\nu_k}{d}$.  Our Theorem \ref{GamgeqKChooseD} extends this set of functions, giving

\[ \sum_{n=0}^\infty q^n \sum_{\lambda \vdash n} \binom{\gamma_{\geq k}(\lambda)}{d} = \frac{1}{(q)_\infty} \prod_{i=1}^{d} \frac{q^{ik}}{1-q^{ik}}. \]

This can be extended to any collection of multiplicities: our Theorem \ref{ManyMults} gives a procedure which algorithmically yields

\[ \sum_{n=0}^\infty q^n \sum_{\lambda \vdash n} \binom{\gamma_{\geq k_1}(\lambda)}{j_1} \binom{\gamma_{\geq k_2}(\lambda)}{j_2} \dots . \]

We execute this for the coefficient $\binom{\gamma_{\geq 1}(\lambda)}{1} \binom{\gamma_{\geq 2}(\lambda)}{1}$.

Let $f_k(n,a)$ be the number of partitions of $n$ with exactly $a$ hooks of size $k$.  In \cite{HookK}, Han gives the generating function

\begin{proposition}\label{HanProp}$$
F_k(q,u) := \sum_{n=0}^\infty \sum_{a=0}^{\infty} q^n u^a f_k(n,a) = \frac{1}{(q)_\infty} \prod_{i=1}^{\infty} (1+q^{ki}(u-1))^k .$$
\end{proposition}

Let ${\mathcal{H}}_i(\lambda)$ be the number of hooks of size $i$ in the partition $\lambda$. Differentiation and specialization of the above identity at $u=1$ yields the generating functions $\sum_{n=0}^\infty q^n \sum_{\lambda \vdash n} \binom{{\mathcal{H}}_i(\lambda)}{d}$.  By a combinatorial mapping we are able to establish a family of identities for the joint counts of hooks of size 1 and 2, including Corollary \ref{H1H2} of Theorem \ref{ManyMults}, the identity

\[ \sum_{n=0}^\infty q^n \sum_{\lambda \vdash n} \binom{{\mathcal{H}}_1(\lambda)}{1} \binom{{\mathcal{H}}_2(\lambda)}{1} = 2 \frac{1}{(q)_\infty} \frac{q^2+q^4+q^5}{(1-q^2)(1-q^3)}. \]

These results are established in order to prove a conjecture from the earlier paper \cite{Keith}, which initially motivated these investigations.  Although less than a complete generating function, we think that this result is suggestive of further utility in this line of investigation, and the combinatorial observations involved are interesting in their own right.  In particular, the result above is necessary to establish the following relation, which is the heart of Theorem \ref{QuadEq}:

\begin{align*}
&\qquad \sum_{\lambda \vdash n} \bigg[  \binom{{\mathcal{H}}_{1}}{2} + \frac{1}{4} \binom{{\mathcal{H}}_{1}}{1}\binom{{\mathcal{H}}_{2}}{1} + \frac{1}{16}\binom{{\mathcal{H}}_{2}}{2} \bigg] \\
&= \sum_{\lambda \vdash n} \bigg[  \binom{\gamma_{\geq 1}}{2} + \frac{1}{2} \binom{\gamma_{\geq 1}}{1}\binom{\gamma_{\geq 2}}{1} + \frac{1}{4}\binom{\gamma_{\geq 2}}{2} \bigg] + \frac{1}{16}\gamma_{\geq 4}(n)
\end{align*}

This in turn is a statement which relates the 2-hook truncation of the Han/Nekrasov-Okounkov hooklength formula to a 2-finitized version of the generating function for the $(k,j)$-colored partitions, which we define more completely in the next section.  Every such truncation should yield a similar identity.

In the next section of the paper we give all necessary definitions and notations.  In Section \ref{sec:Constant and Linear Equivalence} we prove an infinite family of relations on the linear truncations of the series involved.  In Section \ref{sec:Quadratic} we prove Theorems \ref{GamgeqKChooseD} through \ref{QuadEq}, establishing relations at the $q^2$ terms of the underlying series.  In the final section we conclude with comments and consideration of what open questions seem to hold the most potential from here.

\section{Definitions}\label{sec:Definitions}

This section will formalize the notation used in the Introduction and throughout the rest of the paper. Most formal definitions here will be the same as those in~\cite{Keith}, for reference.  The standard reference for partition theory is \cite{AndPtn}.

A partition of $n$ is a weakly decreasing sequence of positive integers $\lambda_i$ which sums to $n$, given by $\lambda = (\lambda_1, \dots, \lambda_j)$. We also use the frequency notation $\lambda = 1^{\nu_1}2^{\nu_2}\dots$ to indicate a partition in which there are $\nu_1$ parts of size 1, $\nu_2$ parts of size 2, etc.  Let $\lambda \vdash n$ denote that $\lambda$ partitions $n$. The number of partitions of $n$, $p(n)$, has the generating function

\begin{align*}
P(q) &= \sum_{n=0}^\infty p(n)q^n = \prod_{n=1}^{\infty} \frac{1}{1-q^n}.
\end{align*}

When relevant, we will use the standard notation for the $q$-Pochhammer symbol, a.k.a.\ the $q$-shifted factorial:

\begin{align*}
(a;q)_n &= (1-aq^0)(1-aq^1)\dots(1-aq^{n-1}), \\
(a;q)_{\infty} &= \lim_{n\rightarrow \infty} (a;q)_n \\
(q)_\infty &:= (q;q)_\infty = \prod_{k=1}^{\infty} (1-q^k).
\end{align*}

Then $(q)_\infty^{-1} = P(q)$.

The set of $k$-colored whole numbers is $\mathbb{N}_k = \{a_b | a,b\in\mathbb{N}, 1\leq b\leq k\}$, with magnitude $\vert a_b \vert = a$ and order $a_b < c_d$ if $a < c$ or (($a=c$) AND ($b < d$)). A $k$-colored partition of an integer $n$ is a weakly decreasing sequence $\lambda=(\lambda_1,\dots,\lambda_l)$ with all $\lambda_i \in \mathbb{N}_k$ such that the sum of the $\lambda_i$ is $n$.

The $k$-colored partitions are partitions in which each part may be assigned one of $k$ available colors, with the order of said colors not mattering. By convention, we denote the colors of a partition as subscripts, listed in weakly decreasing order for parts of the same size. The 2-colored partitions of 3 are
\begin{align*}
&3_2, 3_1, \\
&2_2 + 1_2, 2_2 + 1_1, 2_1 + 1_2, 2_1 + 1_1, \\
&1_2 + 1_2 + 1_2, 1_2 + 1_2 + 1_1, 1_2 + 1_1 + 1_1, 1_1 + 1_1 + 1_1.
\end{align*}

The generating function for the number of $k$-colored partitions of $n$, $c_k(n)$, is formed by raising the generating function for a basic partition to the $k^{th}$ power.

\begin{displaymath}
C_k(q) :=
\sum_{n=1}^{\infty} c_k(n) q^n =
\prod_{n=1}^{\infty} \frac{1}{(1-q^n)^k}.
\end{displaymath}

Overpartitions are partitions in which the last part of a given size is either marked, or not. The overpartitions of 3 are $$3, \overline{3}, 2+1, 2+\overline{1}, \overline{2}+1, \overline{2}+\overline{1}, 1+1+1, 1+1+\overline{1}.$$

The generating function for $\overline{p}(n)$, the number of overpartitions of n, is:
\[
\overline{P}(q) := \sum_{n=0}^\infty \overline{p}(n) q^n
= \prod_{k=1}^\infty \frac{1+q^k}{1-q^k}
= \prod_{k=1}^\infty \frac{1-q^{2k}}{(1-q^k)^2}.
\]

Overpartitions are an active area of research, with authors including Corteel \& Lovejoy, Andrews, and others (\cite{CoLo}, \cite{Andrews}, etc.).

In the language of $k$-colored partitions, overpartitions can be considered 2-colored partitions in which only one color is allowed per size of part. From this viewpoint, in~\cite{Keith}, the second author defined the $(k,j)$-colored partitions: partitions in which at most $j$ of $k$ available colors can appear for a given size of part. That paper gave the following generating function for $c_{k,j}(n)$, the number of such partitions:

\begin{align*}
C_{k,j}(q) :=
\sum_{n=0}^{\infty} c_{k,j}(n)q^n &=
\prod_{n=1}^{\infty}
\bigg(1 + \frac{ \binom{k}{1} q^n }{1-q^n} + \frac{ \binom{k}{2} q^{2n} }{(1-q^n)^{2n}} + \dots + \frac{ \binom{k}{j} q^{jn} }{ (1-q^n)^j }\bigg) \\ &=
\frac{1}{{(q)_\infty}^j} \prod_{n=1}^{\infty}
\bigg(
\sum_{i=0}^{j} \binom{k}{i} (1-q^n)^{j-i} q^{in}
\bigg).
\end{align*}

Overpartitions are then the $(2,1)$-colored partitions.

If in $C_{k,j}(q)$ we let $k=1-b$ and let $j$ increase without bound, we obtain the following for $C_{1-b,\infty}(q)$:

\begin{align*}
C_{1-b,\infty}(q) &:=
\prod_{n=1}^{\infty}
\sum_{i=0}^{\infty}
\binom{1-b}{i} \frac{q^{in}}{(1-q^n)^i} =
\prod_{n=1}^{\infty} \bigg( 1 + \frac{q^n}{1-q^n} \bigg) ^ {1-b} \\
&= \prod_{n=1}^{\infty} \bigg( \frac{1}{1-q^n} \bigg)^{1-b} =
\prod_{n=1}^{\infty} (1-q^{n})^{b-1}.
\end{align*}

If we expand this expression with indefinite $j$ we obtain a formula dependent on the multiplicities of parts in partitions $\lambda$.  As we shall see later, a natural truncation is to treat any multiplicity greater than $j$ as simply $j$.

The Ferrers diagram of a partition $(\lambda_1, \dots, \lambda_r)$ is a stack of unit-size squares justified to the origin in the fourth quadrant. For example, the Ferrers diagram of a partition of 20, $\lambda=(5,4,3,3,2,2,1)$, is
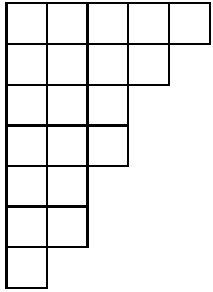
\begin{displaymath}
\ytableaushort[*(white)]{~~~~~,~~~~,~~~,~~~,~~,~~,~}
\end{displaymath}

The hook length $h_{ij}$ of the square with lower right corner at $(-i,-j)$ in the plane is the count of the number of squares both to its right and directly below it in the Ferrers diagram, including itself. So, the hook length of the $(-2,-2)$ square in the Ferrers diagram above is 7. This has been illustrated below, filling in the other hook lengths for reference.

\begin{displaymath}
\ytableaushort[*(white)]{{11}9631,
9{*(gray)7}{*(lightgray)4}{*(lightgray)1},
7{*(lightgray)5}2,
6{*(lightgray)4}1,
4{*(lightgray)2},
3{*(lightgray)1},
1}
\end{displaymath}

The formula for $C_{1-b,\infty}$ given above is precisely the formula considered by Guo-Niu Han and Nekrasov \& Okounkov in their famous hooklength formula giving the coefficients on $q^n$ as polynomials in the complex indeterminate $b$:

\begin{displaymath}
HNO(q) :=
\sum_{n=0}^{\infty} p_n (b)q^n :=
\prod_{n=1}^{\infty} (1-q^n)^{b-1} =
\sum_{n=0}^{\infty} q^n \sum_{\lambda \vdash n} \prod_{h_{ij}\in\lambda}
(1-\frac{b}{h_{i,j}^2})
\end{displaymath}

\noindent where the $h_{ij}$ are the hooklengths that appear in the Ferrers diagram of a partition $\lambda$ of $n$.

Since the two series are equal in the infinite limit of the first, it is natural to ask about intermediate finite cases.  Two natural truncations of each function are to set $j$ to be a finite value in $C_{1-b,j}$ and to restrict the Han-Nekrasov/Okounkov fomula, hereinafter $HNO$, to consider hooks of size at most $j$.  

We denote the truncated hooklength formula by $HNO_j(q)$:

\[
HNO_j(q) := \sum_{n=0}^{\infty} q^n \sum_{\lambda \vdash n} \prod\limits_{\substack{h_{i,k}\in\lambda \\ h_{i,k}\leq j}} (1 - \frac{b}{h_{i,k}^2}) .
\]

This paper will explore the combinatorial and algebraic relationship of the two formulae under truncation by $j$, expanding upon previous conjectures. 

We define ${\mathcal{H}}_{j}(\lambda)$ as the number of hooks $h_{i,k}=j$ in a partition $\lambda$. If $\lambda$ is clear from context, we shorten this to ${\mathcal{H}}_{j}$. Also, $\sum_{\lambda \vdash n} {\mathcal{H}}_{j} = {\mathcal{H}}_{j}(n)$, the total number of hooks of size $j$ in all partitions of $n$.

The conjugate of a partition $\lambda$, denoted by $\lambda'$, is the partition of $\lambda$ reflected across the diagonal (in our plane description, across the line $y=-x$). A partition fixed under conjugation is a self-conjugate partition. The conjugate of $\lambda=(5,4,3,3,2,2,1)$ is $\lambda'=(7,6,4,2,1)$, illustrated below. The hooklengths have again been filled in to demonstrate how the number of hooks of size $k$ in $\lambda$ is the same as in its conjugate.

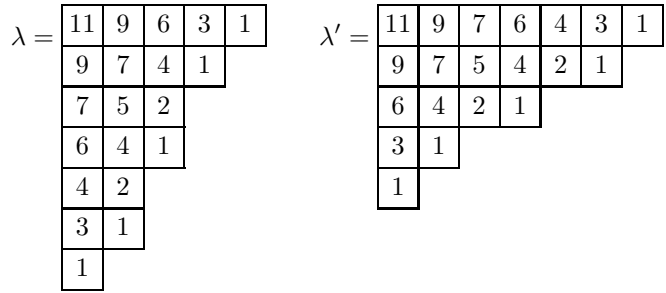
\begin{align*}
\lambda = \ytableaushort[*(white)]{{11}9631,9741,752,641,42,31,1} \qquad
\lambda' = \ytableaushort[*(white)]{{11}976431,975421,6421,31,1}
\end{align*}

We define the following vectors: $\lambda = (\lambda_1 , \dots, \lambda_t)$ itself, the $n$-dimensional vector representing our partition and as previously defined; $\nu$ as in Bacher and Manivel ~\cite{BM}, where $\nu_i$ counts the multiplicity of parts of size $i$ in the partition $\lambda=1^{\nu_1}2^{\nu_2}\dots$; and the vector $\gamma$, where $\gamma_j$ counts $ \vert \{ \nu_i = j \} \vert$ in $\nu$. This can be thought of as a vector counting the "multiplicity of multiplicities":

\begin{align*}
\nu &= (\nu_1, \nu_2, \dots, \nu_n) , \quad \nu_i = \vert \{ \lambda_j = i \} \vert \\
\nu(n) &= (\nu_1(n), \nu_2(n), \dots, \nu_n(n)) = \sum_{\lambda \vdash n} \nu (\lambda) \\
\gamma &= (\gamma_1, \gamma_2, \dots, \gamma_n) , \quad \gamma_i = \vert \{ \nu_j = i \} \vert \\
\gamma(n) &= (\gamma_1(n), \gamma_2(n), \dots, \gamma_n(n)) = \sum_{\lambda \vdash n} \gamma (\lambda)
\end{align*}

\noindent We also define $\gamma_{\geq k} = \gamma_k + \gamma_{k+1} + \dots $.

Bacher and Manivel prove a multitude of theorems regarding these vectors.  Important for us in this paper will be:

\begin{align*}
\nu_k(n) = \sum_{i=k}^n \gamma_i(n) = \gamma_{\geq k}.
\end{align*}

They prove several theorems regarding $d$-th moments of $\lambda_k$ and $\nu_k$, particularly generating functions. We extend this list by providing the generating function for the $d$-th moment of $\gamma_{\geq k}$.

\section{Constant and Linear Equivalence}\label{sec:Constant and Linear Equivalence}

We begin with a useful expression for the truncation of $C_{1-b,j}(q)$. 

\begin{theorem} Construct ${C'}_{1-b,j}$ by expanding $C_{1-b,j}$ over all partitions and truncate by considering any multiplicity greater than $j$ to be $j$.  The resulting truncation has the formula 
\begin{align*}
{C'}_{1-b,j}(q) := \sum\limits_{n=0}^{\infty} c_{1-b,j}(n) q^n &=
\sum\limits_{n=0}^{\infty} q^n
\sum_{\lambda \vdash n}
\prod\limits_{\nu_i} \binom{\min(j,\nu_i) - b}{\min(j,\nu_i)}.
\end{align*}
\end{theorem}

\begin{proof}
Begin with the following form of $C_{1-b,j}(q)$ from~\cite{Keith}: 

\begin{align*}
C_{1-b,j}(q) &= \prod_{n=1}^{\infty} \bigg[ \frac{1}{(1-q^n)^j} \sum_{i=0}^{j} \binom{1-b}{i} (1-q^n)^{j-i} q^{in} \bigg].
\end{align*}

From this, we use the Binomial Theorem $\frac{1}{(1-q^n)^j} = \sum_{p=0}^{\infty} \binom{j+p-1}{j-1} q^{pn}$ and the identity $\sum_{i=0} (-1)^i \binom{r}{i} \binom{j-i}{m-i} = \binom{j-r}{m}$:

\[
C_{1-b,j}(q) = \prod_{n=1}^{\infty}\bigg[\sum_{p=0}^{\infty} \sum_{k=0}^{j}(-1)^k \binom{j+p-1}{j-1}  \binom{b+j-1}{k} q^{(p+k)n}\bigg].
\]

Now take the product of these series over all $n$ and consider the expression's contribution to $\lambda\vdash N = 1^{\nu_1} 2^{\nu_2} 3^{\nu_3} \dots$ in the claimed equality.  We have that $(p+k)n = \nu_n n$ and so we have contribution 

\begin{multline*}
\prod\limits_{\nu_i =1}   \big[ \binom{j+1-1}{j-1} - \binom{j+0-1}{j-1} \binom{b+j-1}{1}\big] \\
* \prod\limits_{\nu_i =2}
\big[ \binom{j+2-1}{j-1} - \binom{j+1-1}{j-1} \binom{b+j-1}{1} + \binom{j+0-1}{j-1} \binom{b+j-1}{2} \big] \\
*\dots
* \prod\limits_{\nu_i \geq j}
\big[ \sum\limits_{\ell=0}^{j} (-1)^\ell \binom{j-1+\nu_i -\ell}{j-1} \binom{b+j-1}{\ell} \big].
\end{multline*}

Standard identities and a little algebra now give

\begin{align*}
\sum\limits_{\ell=0}^{\nu_i} (-1)^\ell \binom{j-1+\nu_i -\ell}{j-1} \binom{b+j-1}{\ell}
&=
  \sum\limits_{\ell=0}^{\nu_i} (-1)^\ell \binom{j-1+\nu_i -\ell}{\nu_i - \ell} \binom{b+j-1}{\ell} \\
&=
  \sum\limits_{\ell=0}^{\nu_i} (-1)^\ell \binom{b+j-1}{\ell} \binom{(j-1+\nu_i) -\ell}{(\nu_i) - \ell} \\
&= (-1)^{\nu_i} \sum_{\ell=0}^{\nu_i} \binom{b+j-1}{\ell} \binom{-j}{\nu_i - \ell}
\\  &= (-1)^{\nu_i} \binom{b-1}{\nu_i}
\\  &= \binom{\nu_i - b}{\nu_i}.
\end{align*}

The inner sum does not thus simplify, so we invoke the truncation referenced in the theorem.  The final products now all become $\binom{j-b}{j}$.  We obtain

\begin{multline*}
\sum\limits_{N=0}^{\infty} q^N \sum\limits_{\substack{\lambda \vdash N \\ \lambda= 1^{\nu_1}2^{\nu_2}\dots}}\prod\limits_{\nu_j \leq j}
\binom{\nu_j - b}{\nu_j}\prod\limits_{\nu_j > j}
   \binom{j - b}{j} \\
=\sum\limits_{N=0}^{\infty} q^N\sum\limits_{\substack{\lambda \vdash N \\ \lambda= 1^{\nu_1}2^{\nu_2}\dots}} \prod\limits_{\nu_j}
\binom{\min(j,\nu_j) - b}{\min(j,\nu_j)}.
\end{multline*} \end{proof}

\noindent \textbf{Remark.} We note that this is in fact a truncation, since as $j \rightarrow \infty$, entries in the sequence ${C'}_{1-b,j}$ become equal to $C_{1-b,\infty}$ in all coefficients of $q^n$ for increasing $n$.

With this simplified generating function, we can more easily extract the coefficient on each $b^c$ term in the polynomial coefficient on $q^n$.

\begin{theorem}\label{coeff}
The coefficient on the $b^c$ term in the polynomial coefficient of $q^n$ of ${C'}_{1-b,j}(q)$ and $HNO_j(q)$ are as follows:
\begin{align*}
\left[b^{c}\right] \left[q^n\right] {C'}_{1-b,j} (q) &= \sum_{\lambda \vdash n} \sum_{a_1 + \dots + a_j = c} (\frac{1}{1})^{a_1} \binom{\gamma_{\geq 1}}{a_1} * \dots * (\frac{1}{j})^{a_j} \binom{\gamma_{\geq j}}{a_j}
\\
\left[b^{c}\right] \left[q^n\right] HNO_j (q) &= \sum_{\lambda \vdash n} \sum_{a_1 + \dots + a_j = c} (\frac{1}{1^2})^{a_1} \binom{{\mathcal{H}}_{1}}{a_1} * \dots * (\frac{1}{j^2})^{a_j} \binom{{\mathcal{H}}_{j}}{a_j}.
\end{align*}
\end{theorem}

\begin{proof}
The $b^c$ coefficients in $HNO_j(q)$ are given by the binomial theorem. For ${C'}_{1-b,j}(q)$ we consider its expansion and manipulate it into a similar form.

\begin{align*}
\left[q^n\right] C_{1-b,j} (q,b) &= \sum_{\lambda \vdash n} \binom{1-b}{1}^{\gamma_1} \binom{2-b}{2}^{\gamma_2} \dots \binom{j-b}{j}^{\gamma_{\geq j}}\\
&= \sum_{\lambda \vdash n} (1-b)^{\gamma_1} (\frac{1}{2!}(1-b)(2-b))^{\gamma_2} \dots (\frac{1}{j!}(1-b)(2-b)\dots(j-b))^{\gamma_{\geq j}}\\
&= \sum_{\lambda \vdash n} (\frac{1}{1}(1-b))^{\gamma_{\geq 1}} (\frac{1}{2}(2-b))^{\gamma_{\geq 2}} \dots (\frac{1}{j}(j-b))^{\gamma_{\geq j}}
\end{align*}

From here, we again apply the binomial theorem.\\
\end{proof}

We now consider relationships between these two expressions, which will naturally give rise to observations on relations between part sizes, multiplicities, and hooklengths which further those of Han and Bacher \& Manivel.

For the truncation at $j=1$, it was shown in \cite{Keith} that the two truncations are in fact equal, i.e. $${C'}_{1-b,1}(q) = HNO_1(q).$$  This is immediate from the combinatorial statement that nonzero multiplicities, and hooks of size exactly 1, are both in equal in number in any partition to the number of distinct part sizes that appear.

At $j=2$, both ${C'}_{1-b,2}(q)$ and $HNO_2(q)$ have the same constant and linear term in $b$, by Theorem 8 of~\cite{Keith} via a direct combinatorial mapping. It was conjectured in that paper that the two formulae match at the constant and linear terms for all $j$.  We here show this.

\begin{theorem}[Constant \& Linear Term Equivalence]
For all $j$ and all $i$, the polynomial coeffficient of $q^i$ in $HNO_j(q)$ and ${C'}_{1-b,j}(q)$ have the same constant and linear term in $b$.
\end{theorem}

\begin{proof}
The proof is straightforward using Theorem~\ref{coeff}. The constant term $c=0$ is simply $p(n)$, as we get a sum of 1 over the partitions of $n$ for both. For the linear term, expand both expressions at $c=1$:

\begin{align*}
\left[b^{1}\right] \left[q^n\right] {C'}_{1-b,j} (q,b) &= \sum_{\lambda \vdash n} \sum_{a_1 + \dots + a_j = 1} (\frac{1}{1})^{a_1} \binom{\gamma_{\geq 1}}{a_1} * \dots * (\frac{1}{j})^{a_j} \binom{\gamma_{\geq j}}{a_j} \\
&= \sum_{\lambda \vdash n} \frac{1}{1} \gamma_{\geq 1} + \dots + \frac{1}{j} \gamma_{\geq j} \\\\
\left[b^{1}\right] \left[q^n\right] HNO_j (q,b) &= \sum_{\lambda \vdash n} \sum_{a_1 + \dots + a_j = 1} (\frac{1}{1^2})^{a_1} \binom{{\mathcal{H}}_{1}}{a_1} * \dots * (\frac{1}{j^2})^{a_j} \binom{{\mathcal{H}}_{j}}{a_j} \\
&= \sum_{\lambda \vdash n}  \frac{1}{1^2} {\mathcal{H}}_{1} + \dots + \frac{1}{j^2} {\mathcal{H}}_{j} \\
\end{align*}
The claimed identity now follows from the termwise equality, for $0\leq i \leq j$:

\begin{align*}
\sum_{\lambda \vdash n} \frac{1}{i^2} {\mathcal{H}}_{i} &\stackrel{?}{=} \sum_{\lambda \vdash n} \frac{1}{i} \gamma_{\geq i} \\
\mathcal{H}_i(n) = \sum_{\lambda \vdash n} {\mathcal{H}}_{i} &\stackrel{?}{=} i \sum_{\lambda \vdash n} \gamma_{\geq i} = i * \nu_i
\end{align*}

The identities $\gamma_{\geq i} (n) = \nu_i (n)$ and $\mathcal{H}_i(n) = i*\nu_i$ are well-known, as mentioned in the Introduction, and we are done.
\end{proof}

\section{Quadratic Equivalence}\label{sec:Quadratic}

The next natural step is to consider what the difference is in the quadratic terms of $HNO_j(q)$ and ${C'}_{1-b,j}(q)$, given that they are not equal. We will begin with $j=2$.
Using Theorem~\ref{coeff}, we can determine the quadratic coefficient on the $q^n$ term of both $H_2(q)$ and ${C'}_{1-b,2}(q)$:

\begin{align*}
\left[b^{2}\right] \left[q^n\right] {C'}_{1-b,2} (q) &= \sum_{\lambda \vdash n} \bigg[  \binom{\gamma_{\geq 1}}{2} + \frac{1}{2} \binom{\gamma_{\geq 1}}{1}\binom{\gamma_{\geq 2}}{1} + \frac{1}{4}\binom{\gamma_{\geq 2}}{2} \bigg] \\
\left[b^{2}\right] \left[q^n\right] H_2 (q) &= \sum_{\lambda \vdash n} \bigg[  \binom{{\mathcal{H}}_{1}}{2} + \frac{1}{4} \binom{{\mathcal{H}}_{1}}{1}\binom{{\mathcal{H}}_{2}}{1} + \frac{1}{16}\binom{{\mathcal{H}}_{2}}{2} \bigg]
\end{align*}

The second author conjectured in~\cite{Keith} that the exact term one needs to add to the $b^2$ coefficient on each $q^n$ of ${C'}_{1-b,2}(q)$ to obtain that of $H_2(q)$ is $\frac{1}{16}\gamma_{\geq 4}(n)$. We will prove this.

We first give the bivariate generating function $G_{k}(q,u)$ defined by Bacher and Manivel in~\cite{BM}.  Let $g_k(n,t)$ be the number of partitions of $n$ with exactly $t$ part sizes having multiplicity at least $k$.  We have

\begin{align*}
G_k(q,u) &:= \sum_{n=0}^\infty \sum_{t=0}^\infty g_k(n,t) q^n u^t \\
&= \prod_{i=1}^\infty (1+q^i+q^{2i}+\dots+q^{(k-1)i} + u(q^{ki} + q^{(k+1)i} + \dots)) \\
&= \prod_{i=1}^\infty \bigg( \frac{1-q^{ki}}{1-q^i} + u \frac{q^{ki}}{1-q^i} \bigg) = \frac{1}{(q)_\infty} \prod_{i=1}^\infty (1 - q^{ki} + uq^{ki}) \\
&= \frac{1}{(q)_\infty} \prod_{i=1}^\infty (1 + (u-1)q^{ki}).
\end{align*}

We can now obtain the generating function for $\sum_{\lambda \vdash n}\binom{\gamma_{\geq k} (\lambda)}{d}$:

\begin{theorem}\label{GamgeqKChooseD} With $G_{k}(q,u)$ as above,
\begin{align*}
\sum_{n=0}^\infty q^n \sum_{\lambda \vdash n} \binom{\gamma_{\geq k}(\lambda)}{d}
&= \frac{1}{d!} \pdv[d]{}{u} G_k(q,u) \bigg\rvert_{u=1}
&= \frac{1}{d!} \pdv[d]{}{u} \bigg[ \frac{1}{(q)_\infty} \prod_{i=1}^{\infty}(1+(u-1)q^{ki})\bigg] \bigg\rvert_{u=1} \\
= \frac{1}{(q)_\infty} \prod_{i=1}^{d} \frac{q^{ik}}{1-q^{ik}}.
\end{align*}
\end{theorem}

\begin{proof}
Recall that given a partition $\lambda$, we denote by $\gamma_{\geq k}$ the number of part sizes with multiplicity at least $k$ in $\lambda$. Taking the derivative of $G_k(q,u)$ with respect to $u$ and setting $u$ equal to 1 yields $\gamma_{\geq k}(\lambda)$ summed over all $\lambda$ of $n$ as the coefficient on $q^n$.
\[
\pdv{}{u}G_k(q,u) \bigg\rvert_{u=1} = \sum_{n=0}^\infty q^n \sum_{\lambda \vdash n} \gamma_{\geq k} (\lambda) .
\]
Then $\pdv[d]{}{u}G_k(q,u) \big\rvert_{u=1}$ yields $\gamma_{\geq k}(\lambda)(\gamma_{\geq k}(\lambda)-1)\dots(\gamma_{\geq k}(\lambda)-d)$ summed over the partitions of $n$ as the coefficient on $q^n$, so we have the following:
\begin{align*}
\pdv[d]{}{u}G_k(q,u) \bigg\rvert_{u=1} &= \sum_{n=0}^\infty q^n \sum_{\lambda \vdash n} \gamma_{\geq k}(\lambda)(\gamma_{\geq k}(\lambda)-1)\dots(\gamma_{\geq k}(\lambda)-d) \\
&= d! \sum_{n=0}^\infty q^n \sum_{\lambda \vdash n} \binom{\gamma_{\geq k}(\lambda)}{d} \\
\end{align*}

Now consider repeated derivatives of $\prod_{i=1}^{\infty} (1+(u-1)q^{ki})$ with respect to $u$. Using the product rule, we get a sum over all $i$ of $q^{ik}$ times the original product, with the $i$th term removed. Doing this for each derivative, we gain a new $q^{ik}$ term each time, removing these from the infinite product. Thus, we have:

\begin{align*}
\pdv[d]{}{u} \bigg[ \prod_{i=1}^{\infty} (1+(u-1)q^{ik}) \bigg] \bigg\rvert_{u=1}
&= \sum\limits_{\substack{(x_1,\dots,x_d) \in \mathbb{N}^d \\ x_i \neq x_j}} q^{k \sum x_i} \frac{\prod_{j=1}^{\infty}(1+(u-1)q^{kj})}{\prod_{x_i} (1+(u-1)q^{ki})} \bigg\rvert_{u=1}\\
&= \sum\limits_{\substack{(x_1,\dots,x_d) \in \mathbb{N}^d \\ x_i \neq x_j}} q^{k \sum x_i}
\end{align*}

This is the generating function for the number of compositions into exactly $d$ distinct parts, magnified by a factor of $k$. Every composition into $d$ distinct parts can be reordered into a partition into distinct parts, with $d!$ compositions corresponding to a single partition. This gives the following:

\begin{align*}
\sum\limits_{\substack{(x_1,\dots,x_d) \in \mathbb{N}^d \\ x_i \neq x_j}} q^{k \sum x_i} &= d! \sum\limits_{\substack{\text{partitions } (y_1,\dots,y_d)\\ \text{into } d \text{ distinct parts}}} q^{k \sum{y_i}} \\
&= d! \frac{q^{k \binom{d+1}{2}}}{ \prod_{i=1}^{d} (1-q^{ki})} = d! \prod_{i=1}^{d} \frac{q^{ki}}{1-q^{ki}}.
\end{align*}

Dividing through by $d!$ gives us our generating function for $\sum_{n=0}^\infty q^n \sum_{\lambda \vdash n} \binom{\gamma_{\geq k}}{d}$.

\end{proof}

From Han's generating function Proposition \ref{HanProp} for hooks of a given length we can obtain the generating functions for $\binom{{\mathcal{H}}_k}{d}$ for any given $k$ and $d$.  In particular, we need the following:

\begin{lemma}
\begin{align*}
\sum_{n=0}^\infty q^n \sum_{\lambda \vdash n} \binom{{\mathcal{H}}_{k}}{2} &= \frac{1}{2}\frac{1}{(q)_\infty} \bigg[\bigg(\frac{kq^{k}}{(1-q^{k})}\bigg) ^2 - \frac{kq^{2k}}{1-q^{2k}}\bigg] \\
\text{ and specifically} \\
\sum_{n=0}^\infty q^n \sum_{\lambda \vdash n} \binom{{\mathcal{H}}_{2}}{2} &= \frac{1}{(q)_\infty} \frac{q^4 (1+3q^2)}{(1-q^2)(1-q^4)}.
\end{align*}
\end{lemma}

\begin{proof} Employ the procedure above on Han's series.
\end{proof}

For the mixed term, we have a curious result. In this case, we will construct a 2:1 map from $\sum_{\lambda \vdash n} \binom{{\mathcal{H}}_{1}}{k} \binom{{\mathcal{H}}_{2}}{1}$to $\sum_{\lambda \vdash n} \binom{\gamma_{\geq 1}}{k}\binom{\gamma_{\geq 2}}{1}$.

\begin{theorem}
For $k>1$,
\[
\sum_{\lambda \vdash n} \binom{{\mathcal{H}}_{1}}{k} \binom{{\mathcal{H}}_{2}}{1} = 2 \sum_{\lambda \vdash n} \binom{\gamma_{\geq 1}}{k} \binom{\gamma_{\geq 2}}{1}
\]
\end{theorem}

\begin{proof}

Consider a $k$-tuple of the part sizes of $\lambda$, which is a $k$-tuple of ``part sizes repeated at least once.''  Along with these, choose one part size repeated at least twice.  Denote this part size by marking the hook of size 2 that must exist at the far right end of the next-to-last instance of this part size.  (In the figure in the examples, the shaded square at $(-1,-6)$ has such a hook.)

This selection of $k+1$ boxes in the Ferrers diagram is also perforce a selection of a $k$-tuple of hooks of size 1, along with a hook of size 2.  Pair the original selection of $k+1$ repetitions with this selection of $k+1$ hooks.

Next, take the conjugate $\lambda^{\prime}$ of the original partition, and mark the boxes in conjugate position.  This is a valid choice of hooks but not a valid choice of repetitions, as the hooks of size 2 will no longer be on the ends of their rows but rather on the bottoms of columns, and only outer corners are both.  Hence it is distinct from the previous selection of hooks.  Pair this conjugate choice of hooks with the original selection as well.  The result is a 2:1 matching, and the theorem holds.
\end{proof}

We now have the necessary toolkit to prove the final conjecture in~\cite{Keith}.

\begin{theorem}\label{QuadEq}[Quadratic Equivalence]
For all $n$, the coefficient of $b^2$ in each polynomial coefficient of $q^n$ in $H_2$ exceeds that of ${C'}_{1-b,2}$ by $\frac{1}{16}\gamma_{\geq 4}(n)$.  That is, 
\begin{align*}
\left[b^{2}\right] \left[q^n\right] H_2(q) &= \left[b^2\right] \left[q^n\right]{C'}_{1-b,2}(q) + \frac{1}{16}\gamma_{\geq 4}(n).
\end{align*}
\end{theorem}

\begin{proof}

To begin, we restate the claim with our expansions so far:

\begin{multline*}
\sum_{\lambda \vdash n} \bigg[  \binom{{\mathcal{H}}_{1}}{2} + \frac{1}{4} \binom{{\mathcal{H}}_{1}}{1}\binom{{\mathcal{H}}_{2}}{1} + \frac{1}{16}\binom{{\mathcal{H}}_{2}}{2} \bigg] \\
\stackrel{?}{=} \sum_{\lambda \vdash n} \bigg[  \binom{\gamma_{\geq 1}}{2} + \frac{1}{2} \binom{\gamma_{\geq 1}}{1}\binom{\gamma_{\geq 2}}{1} + \frac{1}{4}\binom{\gamma_{\geq 2}}{2} \bigg] + \frac{1}{16}\gamma_{\geq 4}(n)
\end{multline*}

The first terms of the left-hand and right-hand sides, $\sum_{\lambda \vdash n} \binom{{\mathcal{H}}_{1}}{2}$ and $\sum_{\lambda \vdash n} \binom{\gamma_{\geq 1}}{2}$, are equal, as ${\mathcal{H}}_{1}$ and $\gamma_{\geq 1}$ both count the number of part sizes within a partition. The second terms are equivalent by our previous bijection. Lastly, we consider the generating functions of each remaining term and check that they are indeed equal:

\begin{align*}
\sum_{n=0}^\infty q^n \sum_{\lambda \vdash n} \binom{{\mathcal{H}}_{2}}{2}  &\stackrel{?}{=} \sum_{n=0}^\infty q^n (\sum_{\lambda \vdash n} 4\binom{\gamma_{\geq 2}}{2}) + \sum_{n=0}^\infty q^n \gamma_{\geq 4} (n) \\
\frac{1}{(q)_\infty} \frac{q^4 (1+3q^2)}{(1-q^2)(1-q^4)} &\stackrel{?}{=} \frac{1}{(q)_\infty} \bigg(\frac{4q^{6}}{(1-q^{2})(1-q^{4})} + \frac{q^4}{1-q^4} \bigg) \\
&= \frac{1}{(q)_\infty} \frac{q^4 (1+3q^2)}{(1-q^2)(1-q^4)}
\end{align*}
\end{proof}

At this point we observe that the procedure outlined for Theorem \ref{GamgeqKChooseD} is easily extendable to any desired collection of multiplicities.  Given $k_1 < k_2 < \dots < k_r$, let $g_{k_1,k_2,\dots,k_r}(n,t_1,t_2,\dots,t_r)$ be the number of partitions of $n$ with $t_1$ multiplicities of at least $k_1$, along with $t_2$ multiplicities of at least $k_2$, etc., and set $G_{k_1,k_2,\dots,k_r} (q,u_1,\dots,u_r) := \sum_{n=0}^\infty g_{k_1,k_2,\dots,k_r}(n,t_1,t_2,\dots,t_r) q^n u_1^{t_1} \dots u_r^{t_r}$.  Then we can follow a procedure completely analogous to that of Theorem \ref{GamgeqKChooseD}.

\begin{theorem}\label{ManyMults}
\begin{multline*} \sum_{n=0}^\infty q^n \sum_{\lambda \vdash n} \binom{\gamma_{\geq k_1}(\lambda)}{d_1}\dots \binom{\gamma_{\geq k_r}(\lambda)}{d_r} \\ = \frac{1}{d_1! \dots d_r!} \pdv[d_r]{}{{u_r}} \dots  \pdv[d_1]{}{{u_1}} G_{k_1,\dots,k_r} (q,u_1,\dots,u_r) \bigg\rvert_{u_1=\dots=u_r=1}
\end{multline*}
\end{theorem}

This, combined with our 2:1 mapping, allows us to give the generating function for any $\sum_{n=0}^\infty q^n \sum_{\lambda \vdash n} \binom{{\mathcal{H}}_1}{k}\binom{{\mathcal{H}}_2}{1}$.  For instance, we have, as claimed, 

\begin{corollary}\label{H1H2}
$$\sum_{n=0}^\infty q^n \sum_{\lambda \vdash n} \binom{{\mathcal{H}}_1(\lambda)}{1}\binom{{\mathcal{H}}_2(\lambda)}{1} = 2 \frac{1}{(q)_\infty} \frac{q^2+q^4+q^5}{(1-q^2)(1-q^3)}.$$
\end{corollary}

\section{Open Questions}

The following questions immediately suggest themselves.

\begin{enumerate}
\item Theorem \ref{GamgeqKChooseD} gives us a closed form for the count of $\binom{\gamma_{\geq k}}{d}$ for any $k$ and $d$.  Although any single generating function for $\binom{{\mathcal{H}}_k}{d}$ is constructible from Han's generating function, it would be satisfying to have a similar closed form for $\sum_{n=0}^\infty q^n \sum_{\lambda \vdash n} \binom{{\mathcal{H}}_k (\lambda)}{d}$.
\item Likewise, Theorem \ref{ManyMults} gives the joint counts of any desired set of multiplicity thresholds.  On the hook side, Han has a two-variable generating function for one hook length at a time; is there an analogous simple multivariable form for the distribution of multiple hooklengths?  The similarity between Bacher \& Manivel's $G_k(q,u)$ and Han's $F_k(q,u)$ is tantalizing, but extending $F_k$ to more variables does not appear to be so easy.  Experimental computation seems to suggest that a simple product form is not likely.
\item Further exploration of the truncations $HNO_j$ and ${C'}_{1-b,j}$ for higher $j$ will require mappings, or at least identities, involving hooks of size 3 and more.
\item Further exploration of the truncations $HNO_j$ and ${C'}_{1-b,j}$ for higher degrees on $b$ will involve the joint generating functions of multiple $\binom{{\mathcal{H}}_k}{d}$ for $d \geq 1$, among other terms.
\end{enumerate}

The whole array of these truncations should prove to be a field to mine for relationships like those investigated in this paper.  It is also possible that our ${C'}_{1-b,j}$ is not the combinatorially most interesting possible truncation, although it did arise in a natural way.  Whatever truncation method is chosen, a unified theorem giving the associated identities should be a fascinating statement about the interplay between hooks, part sizes, and frequencies in partitions.

\end{document}